\documentclass[reqno,letterpaper,11pt]{amsart}
\usepackage{amssymb, amsfonts, amsmath}
\usepackage{caption}
\usepackage{amsthm}
\usepackage{fullpage}
\usepackage{float}
\usepackage{url}
\usepackage{tikz}
\usetikzlibrary{arrows}
\newtheorem{theorem}{Theorem}[section]
\newtheorem*{theorem*}{Theorem}
\newtheorem*{corollary*}{Corollary}

\newtheorem{lemma}[theorem]{Lemma}

\newtheorem{corollary}[theorem]{Corollary}

\title{Popular Values of the Largest Prime Divisor Function}
\author{NATHAN MCNEW}
\address{Department of Mathematics, Dartmouth College \\ Hanover, NH 03755 USA}
\email{nathan.g.mcnew@dartmouth.edu}

\date{}

\begin{document}

\begin{abstract}
We consider the distribution of the largest prime divisor of the integers in the interval $[2,x]$, and investigate in particular the mode of this distribution, the prime number(s) which show up most often in this list. In addition to giving an asymptotic formula for this mode as $x$ tends to infinity, we look at the set of those prime numbers which, for some value of $x$, occur most frequently as the largest prime divisor of the integers in the interval $[2,x]$.  We find that many prime numbers never have this property.  We compare the set of ``popular primes,'' those primes which are at some point the mode, to other interesting subsets of the prime numbers.  Finally, we apply the techniques developed to a similar problem which arises in the analysis of factoring algorithms.
\end{abstract}
\maketitle

\section{Introduction}
Let $P(n)$ denote the largest prime divisor of an integer $n \geq 2$.  The distribution of the values of this function as $n $ ranges over the interval $[2,x]$ has been considered by several authors.  Alladi and Erd\H{o}s \cite{AlEr} investigated the average order of $P(n)$ (as well as the average order of the $k$-th largest prime factor) and showed that
\begin{equation}
\frac{1}{x}\sum_{n\leq x} P(n) = \frac{\pi^2x}{12\log x}+O\left(\frac{x}{\log^2 x}\right).
\end{equation}
This fact was later shown by Kemeny \cite{Kemeny} using different methods, and improved upon by De Koninck and Ivi\'c, who showed that there exist constants $d_1,d_2 \ldots$ such that for any $m\geq 1$,
\begin{equation}
\frac{1}{x}\sum_{n\leq x} P(n) = x\left(\frac{d_1}{\log x}+\frac{d_2}{\log^2 x} + \cdots + \frac{d_m}{\log^m x} + O\left(\frac{1}{\log^{m+1} x}\right)\right). \label{improvedmean}
\end{equation}
uniformly in $m$.  Naslund \cite{NaslundMean} worked out the values of the constants in this expression, in particular
\begin{equation}
d_m = \sum_{j=0}^m\frac{(-1)^j\zeta^{(j)}(2)}{j!2^{m+1-j}}.
\end{equation}

The median value, $M(x)$, of $P(n)$ as $n$ ranges over the integers in $[2,x]$ was considered by Selfridge and Wunderlich \cite{SeWu} who noted that $M(x) = x^{\frac{1}{\sqrt{e}}+o(1)}$.  The result itself is much older, however, and was essentially Vinogradov's trick for extending the usefulness of character sums.  Naslund \cite{NaslundMed} shows that this median value is given more accurately by 
\begin{equation}
M(x) = e^{\frac{\gamma -1}{\sqrt{e}}}x^{\frac{1}{\sqrt{e}}}\left(1+\frac{c_1}{\log x} + \frac{c_2}{\log^2 x} + \cdots +\frac{c_m}{\log^m x} +O_m\left(\frac{1}{\log^{m+1} x}\right)\right) \label{median}
\end{equation}
where the $c_i$ are computable constants. 

Note that the median value grows substantially slower than the mean value, which indicates that the distribution is skewed strongly to the right.  De Koninck \cite{koninckmode} shows that a mode of this distribution (note that the mode need not necessarily be unique), corresponding to a prime number which occurs with maximal frequency as the largest prime divisor of the integers in $[2,x]$, grows even slower still, slower than any power of $x$.  More precisely, he shows the mode is given by \begin{equation}
e^{\sqrt{\frac{1}{2}\log x \left(\log \log x+ \log \log \log x+O(1)\right)}} \label{koninckmode}
\end{equation}
though in his result the $O(1)$ term is incorrectly given as being $o(1)$.  In what follows, we will say that a prime $p$ is \textbf{popular on the interval} $\mathbf{[2,x]}$ if no prime occurs more frequently than $p$ as the largest prime divisor of the integers in that interval.  While the asymptotic behaviors of the mean and median values of this distribution, as in \eqref{improvedmean} and \eqref{median}, are well understood, the relative error term in \eqref{koninckmode} is quite large. The primary goal of this paper is to improve \eqref{koninckmode} and in particular give the following asymptotic formula, which we prove in Section \ref{results}.  

\begin{theorem} \label{intro:mode}
If the prime $p$ is popular on the interval $[2,x]$ (i.e., $p$ is a mode of the distribution of the largest prime divisor function for that interval) then $p$ satisfies 
\begin{equation*} 
p = \exp\left\{ \sqrt{\nu(x) \log x }+\frac{1}{4}\left(1-\frac{\nu(x)-3}{2\nu(x)^2 -3\nu(x)+1}\right)\right\}\left(1+ O\left(\left(\frac{\log \log x}{\log x}\right)^{1/4}\right)\right) \end{equation*}
where $\nu(x)$ is the solution to the implicitly defined equation $e^{\nu(x)} = 1+ \sqrt{\nu(x)\log x}-\nu(x)$
and is given approximately by
\begin{equation*}
\nu(x) = \tfrac{1}{2}\log \log x + \tfrac{1}{2}\log \log \log x -\tfrac{1}{2}\log 2 + o(1) 
\end{equation*} as $x \to \infty$.
\end{theorem}
Using this we also give an asymptotic expression for the frequency with which the mode value occurs, improving the approximation given in \cite[Theorem 1]{koninckmode}. 
\begin{theorem} \label{intro:modeheight} If $p$ is popular on the interval $[2,x]$, then the number of integers $n \in [2,x]$ for which $P(n)=p$ is given asymptotically by 
\begin{equation}
\Psi\left(\frac{x}{p},p\right) = \frac{x}{\sqrt{2\pi\log x}}\exp\left\{-2\sqrt{\nu(x)\log x}+ \int_{0}^{\nu(x)}\frac{e^s {-}1}{s}ds+\frac{3\nu(x)}{2}+\gamma +O\left(\frac{1}{\log \log x}\right)\right\}. \label{intromodeheight}
\end{equation}
\end{theorem}

In \cite{KoSw} De Koninck and Sweeney consider further the frequency with which prime numbers occur as the largest prime divisor on the interval $[2,x]$.  They note that for a fixed value of $x$ there exists an initial interval $[2,f(x)]$ of primes, $p$ on which the frequency with which $p=P(n)$, monotonically increases at each prime, an intermediate range, $(f(x),g(x))$ where the behavior is oscillatory, and a final interval $[g(x),x]$ on which it monotonically decreases.  They show that for sufficiently large $x$, $f(x)\leq \sqrt{\log x}$ and $g(x) \geq \sqrt{x}$. Clearly the mode value lies somewhere in the intermediate interval.  The oscillatory behavior and the exact value of the mode depends on the spacing and gaps between the primes near this peak value.  

Somewhat surprisingly one finds that there are primes which are not popular on any interval $[2,x]$,  and experimentally it appears that in fact most primes are not.  We therefore define a prime to be a \textbf{popular prime} if it is popular on an interval $[2,x]$ for some value of $x$.  In Section \ref{popprimes} we investigate further this subset of the primes. Clearly there must be infinitely many popular primes.  We are able to show that there is also a positive proportion of prime numbers which are not popular.  To do this we show that the average prime spacing between popular primes cannot be too small.  We prove a more general result which implies the following bound on their spacing.  
\begin{theorem} 
Given any two sufficiently large consecutive primes, $p<q$, if the gap between them, $q-p$, is less than $0.153\log p$, then $p$ is not a popular prime.
\end{theorem} \label{intro:gpy}
\noindent We then combine this with a consequence of the GPY sieve \cite{GPY} which shows that a positive proportion of prime gaps are smaller than that.

\begin{corollary} \label{intro:pospro}
A positive proportion of primes are not popular.
\end{corollary}

In Section \ref{data} we present data on the prime numbers which, for some value of $x\leq \times 10^{14}$, appear most frequently as the largest prime divisors of the integers in $[2,x]$.  We compare these values to other subsets of the prime numbers, in particular the ``convex primes,''  the set of those prime numbers numbers, $p_n$, which form the vertices of the boundary of the convex hull of the points $(n,p_n)$ in the plane, considered by Pomerance \cite{PNG} and recently by Tutaj \cite{Tu}. Within the range of our computations the convex primes are a subset of the popular primes. 

Finally we apply the methods developed in this paper to another problem which turns out to be closely related to ours, the analysis of the running time of factoring algorithms. A key step in several algorithms for factoring integers (including Dixon's
random squares algorithm, the quadratic sieve and the number field sieve) requires generating a pseudorandom sequence of integers $a_1,a_2,\ldots$ until a subset of the $a_i$'s has product equal to a
square.  Pomerance \cite{Pomsquares} notes that in the
(usually heuristic) analysis of these algorithms one can assume that
the pseudo-random sequence $a_1,a_2,\ldots$ is close enough to random
that one can make predictions using this assumption, and thus the analysis of this step of these algorithms can be captured by the following question.
\medskip

\noindent {\bf Pomerance's Problem.} Select positive integers
$a_1,a_2,\ldots \leq x$ independently at random (each integer is chosen with probability $1/x$)
until some subsequence of the $a_i$'s has product equal to a
square. When this occurs, we say that the sequence has a \textit{square dependence}. What is the expected stopping time of this process?
\medskip

Pomerance \cite{Pommultind} showed for any $\epsilon >0$ that as $x \to \infty$ the probability that this stopping time lies in the interval \[\left[\exp\left\{(1-\epsilon)\sqrt{2\log x \log \log x}\right\},\exp\left\{(1+\epsilon)\sqrt{2\log x \log \log x}\right\}\right]\] tends to 1.  Croot, Granville, Pemantle and Tetali \cite {CGPT} showed that the interval can be taken to be $\left[(\frac{\pi e^{-\gamma}}{4} -\epsilon)\frac{x}{h(x)},(e^{-\gamma} +\epsilon)\frac{x}{h(x)}\right]$ with the same result, where $h(x)$ is the maximum value of the function $\frac{\Psi(x,y)}{\pi(y)}$ taken over $y<x$.  (For simplicity, they find $y_0$ which maximizes $\frac{\psi(x,y)}{y}$, and then consider $\frac{\Psi(x,y_0)}{\pi(y_0)}$.)  They give only the same crude approximation \[\frac{x}{h(x)} = \exp\left\{(1+o(1))\sqrt{2\log x \log \log x}\right\}\] as Pomerance however.  In Section \ref{sec:fast} we analyze the values of $y$ which maximize both $\frac{\Psi(x,y)}{y}$ and $\frac{\Psi(x,y)}{\pi(y)}$, and give the following asymptotic for the function $h(x)$.

\begin{theorem} \label{intro:fastpeak}
For a given value of $x$, the value of $h(x)$, the maximum value of $\frac{\Psi(x,y)}{\pi(y)}$ for $y<x$ is given asymptotically by 
\begin{equation*}
h(x) = \frac{x}{\sqrt{2\pi\log x}}\exp\left\{-2\sqrt{\nu(x)\log x}+ \int_{0}^{\nu(x)}\frac{e^s {-}1}{s}ds+\frac{3\nu(x)}{2}+\gamma+O\left(\frac{1}{\log \log x}\right)\right\},
\end{equation*}
\end{theorem}
\noindent the same expression as \eqref{intromodeheight}.

\section{Smooth Numbers} \label{background}

These results rely on careful estimates for the counts of smooth numbers, those integers whose prime factors are all less than some bound.  In particular a number is called $y$-smooth if all of its prime factors are at most $y$.  We will denote by $\Psi(x,y)$ the count of the $y$-smooth numbers up to $x$.  We are specifically interested in the count of the number of integers up to $x$ whose largest prime factor is the prime $p$.  This count is given by $\Psi\left(\tfrac{x}{p},p\right)$ since each integer up to $x$ whose largest prime divisor is $p$ can be written uniquely as $p$ times a $p$-smooth number that is at most $x/p$.  

The function $\Psi(x,y)$ has been well studied over the course of the last century. From Hildebrand \cite{hild} we know that for each $\epsilon >0$, $x>2$ and $\exp\left((\log \log x)^{5/3+\epsilon}\right)<y<x$, \begin{equation}
\Psi(x,y) =  x\rho(u)\left(1+O_\epsilon\left(\frac{\log(u+1)}{\log y}\right)\right) \label{hildapprox}
\end{equation}
where \begin{equation*}u = \frac{\log x}{\log y}\end{equation*} and $\rho(u)$, the Dickman rho function, is the continuous solution to the differential delay equation \begin{equation} \label{dickmanfneq}
u\rho'(u)+\rho(u-1)=0
\end{equation} with the initial condition $\rho(u)=1$, $(0 \leq u \leq 1)$.
It was shown by Alladi \cite{alladi} that as $u \to \infty$,
\begin{equation}
\rho(u) = \left(1+O\left(\frac{1}{u}\right)\right)\sqrt{\frac{\xi'(u)}{2\pi}} \exp\left\{\gamma -u\xi(u) + \int_{0}^{\xi(u)}\frac{e^s -1}{s}ds\right\}. \label{alladiapprox}
\end{equation}
Here $\gamma$ is the Euler-Mascheroni constant and $\xi(u)$ denotes the unique positive solution to the equation
\begin{equation} \label{xi}
e^{\xi(u)} = 1+u\xi(u)
\end{equation}
which is given approximately by \begin{equation}\xi(u) = \log u + \log \log u +O\left(\frac{\log \log u}{\log u}\right).\end{equation}
It will be useful later to note that 
\begin{equation}
\int_0^u \xi(t) dt = u\xi(u) - \int_0^{\xi(u)} \frac{e^s-1}{s}ds. \label{xiint}
\end{equation}

Saias \cite{saias} gives an approximation for $\Psi(x,y)$ which, while better than Hildebrand's result, is somewhat more cumbersome to work with.  Defining \begin{equation*}\Lambda(x,y) = \begin{cases} \displaystyle{x\int_0^x \rho\left(\frac{\log x - \log t}{\log y}\right)d\frac{\lfloor t \rfloor }{t}} & x \notin \mathbb{Z} \\ 
\displaystyle{\lim_{z\to x^-}\Lambda(z,y)} & x \in \mathbb{Z},\end{cases}\end{equation*}  then the approximation \begin{equation*}
\Psi(x,y) = \Lambda(x,y)\left(1+O_\epsilon\left(\frac{1}{\exp\left((\log y)^{3/5 - \epsilon}\right)}\right)\right) \label{SaiasApprox}
\end{equation*}
holds in the same range as Hildebrand's result.  Assuming the Riemann Hypothesis, this can be improved to 
\begin{equation*}\Psi(x,y) = \Lambda(x,y)\left(1+O_\epsilon\left(\frac{\log x}{ y^{1/2 - \epsilon}}\right)\right).\end{equation*}
Saias also shows that the asymptotic expansion\begin{equation}
\Lambda(x,y) = x\sum_{j=0}^k a_j \frac{\rho^{(j)}(u)}{(\log y)^j} + O_{k,\epsilon}\left(x\frac{\rho^{(k+1)}(u)}{(\log y)^{k+1}}\right) \label{LambdaExpansion}
\end{equation} where the $a_j$ are the coefficients for the Taylor series of $(s-1)\zeta(s)/s$ around $s=1$, holds uniformly for $x \geq 2$, $(\log x)^{1+\epsilon}<y\leq x$ as long as \begin{equation*}\frac{u-j}{k+1-j}\geq \frac{\log \log y}{\log y}\end{equation*} for $0\leq j \leq \min(k,u)$.  We will use extensively Saias' expansion in the case $k=1$.  In particular, the constants $a_0$ and $a_1$ are given by $a_0=1$ and $a_1= \gamma-1$ so that if we define
\begin{equation*}
\kappa(x,y) = \rho(u) + (\gamma-1)\frac{\rho'(u)}{\log y} 
\end{equation*}
then the approximation
\begin{equation}
\Psi(x,y) = x\kappa(x,y)\left(1+O_\epsilon\left(\left(\frac{\log (u+1) }{\log y}\right)^2\right)\right) \label{saiasappone}
\end{equation}
holds in the same range as \eqref{hildapprox}.

In order to make use of Saias' improved approximation we will also require a better approximation of $\rho(u)$.  Both Smida \cite{smida} and Xuan \cite{xuan} have given improved approximations in which the $\left(1+O\left(\frac{1}{u}\right)\right)$ is replaced by a series involving negative powers of $u$ and $\xi(u)$.  Xuan shows that for any fixed integer $N$,
\begin{equation}
\rho(u) = \sqrt{\frac{\xi'(u)}{2\pi}} \exp\left\{\gamma -u\xi(u) + \int_{0}^{\xi(u)}\frac{e^s -1}{s}ds\right\}\left(1+\sum_{i=1}^N \frac{1}{u^i}\sum_{j=0}^\infty \frac{b_{i,j}}{\xi(u)^j} + O_N\left(\frac{1}{u^{N+1}}\right)\right)\label{xuaneq}
\end{equation}
where the $b_{i,j}$ are constants and the series is uniformly convergent.  We will only be using his result in the case that $N=1$.  Smida's work, which is done in greater generality for a family of differential difference equations like Dickman's function, shows that $b_{1,0}= -\frac{1}{12}$.

Finally, Hildebrand \cite[Theorem 3]{H85} gives an upper bound for the number of smooth integers in short intervals which we will useful.  Uniformly for $x>y>2$, $1 \leq z \leq x$ we have 
\begin{align}
\Psi\left(x+z,y\right) -\Psi(x,y) \leq\left(1+O\left(\frac{1}{\log y}\right)\right) \frac{\Psi(x,y)y\log(xy/z)}{\Psi(xy/z,y)\log y}.  \label{psiloc}
\end{align}

\section{Dickman's Function}
The approximation
\begin{equation}\frac{\rho(u-1)}{\rho(u)} = u\xi(u)\left(1+O\left(\frac{1}{u}\right)\right)\label{simpshift} \end{equation}
is common in the literature. (See for example \cite[Section III.5 Corollary  8.3]{tenenbaum}.)  We will need a slightly stronger form obtained using the work of Smida and Xuan.

\begin{lemma} \label{shift} For $u \geq 1$ and any $v \ll 1$ the function $\rho(u)$ satisfies 

\begin{equation}\frac{\rho(u+v)e^{v\xi(u)}}{\rho(u)}=1-\frac{v}{2u}\left(1+\frac{v\xi(u)}{\xi(u)-1} + \frac{1}{\left(\xi(u)-1\right)^2}\right)  +O\left(\frac{1}{u^2}\right).\end{equation}
In particular when $v=-1$,
\begin{equation}\frac{\rho(u-1)}{\rho(u)}=u\xi(u)+\frac{1}{2}+\frac{1}{2(\xi(u)-1)^2} +O\left(\frac{\xi(u)}{u}\right).\end{equation}
\end{lemma}

\begin{proof}
By implicit differentiation of the functional equation $e^{\xi(u)} = 1 +u\xi(u)$ we find that
\begin{align}
\xi'(u) = \frac{\xi(u)}{u\xi(u) -u +1} &=\frac{1}{u}\left(\frac{\xi(u)}{\xi(u) -1}\right)+ O\left(\frac{1}{u^2\xi(u)}\right), \label{xiprime}
\end{align}
that 
\begin{align}\xi''(u) &= \frac{2\xi'(u)-e^{\xi(u)}\xi'(u)^2}{u\xi(u) -u +1}\nonumber \\
 &= \frac{2\xi(u)}{(u\xi(u) -u +1)^2}-\frac{u\xi(u)^3+\xi(u)^2}{(u\xi(u) -u +1)^3} = -\frac{1}{u^2} + O\left(\frac{1}{u^2\xi(u)}\right),\end{align}
and that $\xi'''(u) \sim 1/u^3$. Also,
\begin{align}
\frac{\xi''(u)}{\xi'(u)} &= \frac{2-e^{\xi(u)}\xi'(u)}{\xi(u)-u+1}  = \frac{2}{u\xi(u)-u+1}-  \frac{u\xi(u)^2+\xi(u)}{(u\xi(u)-u+1)^2} \nonumber \\
&= \frac{2}{u}\left(\frac{1}{\xi(u)-1}\right) - \frac{1}{u}\left(\frac{\xi(u)^2}{\left(\xi(u)-1\right)^2}\right)+ O\left(\frac{1}{u^2\xi(u)}\right)\nonumber \\
&=-\frac{1}{u}\left(1 + \frac{1}{\left(\xi(u)-1\right)^2}\right) +O\left(\frac{1}{u^2\xi(u)}\right). \label{xidp}
\end{align}

Now, using equation \eqref{xuaneq} with $N=1$, along with \eqref{xiint} and the approximation 
\begin{align*}
\int_{u}^{u+v} \xi(t)dt = v\xi(u) +\frac{v^2\xi'(u)}{2} + O\left(\left|\xi''(u)\right|\right) 
\end{align*}
we find that 
\begin{align*}
\rho(u+v) &= \rho(u)\sqrt{\frac{\xi'(u+v)}{\xi'(u)}}\exp\left\{-\int_{u}^{u+v} \xi(t)dt\right\}\left(\frac{1+\frac{1}{u+v}\sum_{j=0}^\infty \frac{b_{1,j}}{\xi(u+v)^j} + O\left(\frac{1}{u^2}\right)}{1+\frac{1}{u}\sum_{j=0}^\infty \frac{b_{1,j}}{\xi(u)^j} + O\left(\frac{1}{u^2}\right)}\right)\nonumber \\
&= \rho(u) \sqrt{1 + v\frac{\xi''(u)}{\xi'(u)} + O\left(\frac{\xi'''(u)}{\xi'(u)}\right)}\exp\left\{-v\xi(u) -\frac{v^2\xi'(u)}{2} + O\left(\left|\xi''(u)\right|\right) \right\}\left(1{+}O\left(\frac{1}{u^2}\right)\right)\nonumber \\
&=\rho(u) e^{-v\xi(u)} \left(1 + \frac{v\xi''(u)}{2\xi'(u)} -\frac{v^2\xi'(u)}{2} + O\left(\frac{1}{u^2}\right)\right).
\end{align*}
Here we have used the Taylor expansions for $\sqrt{1+x}$ and $e^x$ around $x=0$.  Finally, using equations \eqref{xiprime} and \eqref{xidp} we have 
\begin{align*}
\rho(u+v)&=\rho(u) e^{-v\xi(u)}\left(1-\frac{v}{2u}\left(1+ \frac{1}{\left(\xi(u)-1\right)^2}\right) -\frac{v^2}{2u}\left(\frac{\xi(u)}{\xi(u)-1}\right) +O\left(\frac{1}{u^2}\right)\right)\nonumber \\
&=\rho(u) e^{-v\xi(u)}\left(1-\frac{v}{2u}\left(1+\frac{v\xi(u)}{\xi(u)-1} + \frac{1}{\left(\xi(u)-1\right)^2}\right)  +O\left(\frac{1}{u^2}\right)\right).
\end{align*}
In the specific case $v=-1$, we have that $e^{\xi(u)} = 1 + u\xi(u)$, and so
\begin{align*}
\rho(u-1) &=\rho(u) \left(u\xi(u) +1\right)\left(1+\frac{1}{2u}\left(1-\frac{\xi(u)}{\xi(u)-1} + \frac{1}{\left(\xi(u)-1)^2\right)}\right)  +O\left(\frac{1}{u^2}\right)\right)\nonumber \\
&=\rho(u) \left(u\xi(u) +1+\frac{\xi(u)}{2}\left(-\frac{1}{\xi(u)-1} + \frac{1}{\left(\xi(u)-1)^2\right)}\right)  +O\left(\frac{\xi(u)}{u}\right)\right)\nonumber \\
&=\rho(u) \left(u\xi(u)+\frac{1}{2}+\frac{1}{2(\xi(u)-1)^2} +O\left(\frac{\xi(u)}{u}\right)\right).
\end{align*}

\end{proof} 

We can use Lemma \ref{shift} to obtain a good approximation for the derivative of $\rho(u)$.

\begin{lemma} \label{rhodlemma}
For $u \geq 1$ we have
\begin{equation}
\rho'(u) = -\rho(u) \left(\xi(u)+\frac{1}{2u}\left(1+\frac{1}{(\xi(u)-1)^2}\right) +O\left(\frac{\xi(u)}{u^2}\right)\right). \label{rhodeq}
\end{equation}
\end{lemma}
\begin{proof}
Using the differential difference equation for $\rho(u)$ and Lemma \ref{shift} we have
\begin{equation*}\rho'(u) = -\frac{1}{u}\rho(u-1) = -\rho(u)\left(\xi(u)+\frac{1}{2u}\left(1+\frac{1}{(\xi(u)-1)^2}\right) +O\left(\frac{\xi(u)}{u^2}\right)\right).\end{equation*}
\end{proof}



\section{The most popular largest prime divisor} \label{results}

For $x \geq 2$ we say that a prime $p$ is popular on the interval $[2,x]$ if no prime occurs more frequently than $p$ as the largest prime divisor of the integers in that interval.  In the case of a tie we will say that any prime which occurs a maximal number of times is popular. The following theorem, Theorem \ref{intro:mode} in the introduction, makes use of Saias' approximation \eqref{saiasappone}.  In particular, this result implies that for each $\epsilon >0$, $x> 4$, $p\geq 2$ and $\exp\left((\log \log x)^{5/3+\epsilon}\right)<p<\frac{x}{p}$, \begin{equation}
\psi\left(\frac{x}{p},p\right) = \frac{x}{p} \kappa\left(\frac{x}{p},p\right)\left(1 + O_{\epsilon}\left(\left(\frac{\log \log x}{\log p}\right)^{2}\right)\right). \label{psixp}
\end{equation} 


\begin{theorem} \label{llqest}
If the prime $p$ is popular on the interval $[2,x]$ then $p$ satisfies 
\begin{equation} 
p = \exp\left\{ \sqrt{\nu(x) \log x }+\frac{1}{4}\left(1-\frac{\nu(x)-3}{2\nu(x)^2 -3\nu(x)+1}\right)\right\}\left(1+ O\left(\left(\frac{\log \log x}{\log x}\right)^{1/4}\right)\right)\label{massympt} \end{equation}
where $\nu(x)$ is the solution to the implicitly defined equation $\nu(x) = \xi\left(\sqrt{\frac{\log x}{\nu(x)}} -1 \right)$
and is given approximately by
\begin{equation}
\nu(x) = \tfrac{1}{2}\log \log x + \tfrac{1}{2}\log \log \log x -\tfrac{1}{2}\log 2 + o(1) 
\end{equation} as $x \to \infty$.
\end{theorem}
\begin{proof}

By using the functional equation \eqref{xi} for $\xi(u)$, we can rewrite the equation for $\nu(x)$ as the solution to \begin{equation} e^{\nu(x)} = 1 + \sqrt{\nu(x)\log x}-\nu(x)\label {nufunceq} \end{equation}
which can be approximated using standard asymptotic techniques to yield the rough approximation above.  

The proof proceeds in three steps, each giving better bounds for any prime that is popular on the interval $[2,x]$.  We show first that as $x \to \infty$, if $p$ is popular on $[2,x]$, then $p$ satisfies \begin{equation}
\exp\left\{\left(\frac{1}{4}+o(1)\right)\sqrt{\frac{\log x}{\nu(x)}}\right\} < p < \exp\left\{(2+o(1))\sqrt{\nu(x)\log x}\right\}. \label{roughapprox}
\end{equation}
Next, we show that 
\begin{equation}
p = \exp\left\{\sqrt{\nu(x)\log x}+O(\log \log x)\right\}, \label{firstapprox}
\end{equation}
and finally that the approximation \eqref{massympt} holds.

To see that $\Psi\left(\frac{x}{p},p\right)$ is maximized near \eqref{roughapprox}, we first set 
\begin{equation*}P_0 = \exp\left\{ \sqrt{\nu(x)\log x }\right\} \end{equation*}
and let \begin{equation*}
u_0 = \frac{\log x}{\log P_0} - 1 = \sqrt{\frac{\log x}{\nu(x)}}-1.\end{equation*}  If $p'$ is the greatest prime less than or equal to $P_0$ then $\Psi\left(\frac{x}{p'},p'\right)\geq \Psi\left(\frac{x}{P_0},P_0\right)$ and hence if $\Psi\left(\frac{x}{q},q\right) < \Psi\left(\frac{x}{P_0},P_0\right)$, for some prime $q$, then $q$ is not popular on $[2,x]$.

Note that by definition $\nu(x) = \xi(u_0)$. We then compute, using \eqref{hildapprox} as well as \eqref{alladiapprox} that \begin{align}
\Psi\left(\frac{x}{P_0},P_0\right) &= \frac{x}{P_0}\rho(u_0)\left(1+O\left(\frac{\log u_0}{\log P_0}\right)\right)\nonumber \\
&=x\sqrt{\frac{\xi'(u_0)}{2\pi}}\exp\left\{\gamma -u_0\xi(u_0) + \int_0^{\xi(u_0)}\frac{e^s-1}{s}ds \right\} \left(1+ O\left(\frac{1}{u_0}\right)\right)\nonumber \\
&\gg \frac{x}{\sqrt{u_0}}\exp\left\{-\left(\sqrt{\frac{\log x}{\nu(x)}}-1\right)\nu(x) -\sqrt{\nu(x)\log x}\right\}\nonumber \\
&= x\exp\left\{ -2\sqrt{\nu(x)\log x}+\nu(x)-\tfrac{1}{2}\log u_0\right\}\nonumber \\
&\geq x \exp\left\{ -2\sqrt{\nu(x)\log x}\right\} \label{peakheight}
\end{align}
for sufficiently large $x$. Using the elementary estimate $\Psi\left(x,y\right) \ll x\exp\left\{-\frac{\log x}{2\log y}\right\}$, $x\geq y\geq 2$, (see \cite[Section III.5 Theorem 1]{tenenbaum}) we see that for any $\epsilon >0$ and sufficiently large $x$ that 
if $q < \exp\left\{\frac{\sqrt{\log x}}{(4+\epsilon)\sqrt{\nu(x)}}\right\}$, then 
\begin{align*}
\Psi\left(\frac{x}{q},q\right) &\ll \frac{x}{q}\exp\left\{-\frac{\log x}{2\log q}\right\}
< x\exp\left\{-(2+\tfrac{\epsilon}{2})\sqrt{\nu(x)\log x} \right\}
\end{align*}
which is asymptotically less than \eqref{peakheight}.  Similarly, if $q> \exp\left\{(2+\epsilon)\sqrt{\nu(x)\log x}\right\}$, then trivially
\begin{align}
\Psi\left(\frac{x}{q},q\right) &<\frac{x}{q} <x\exp\left\{-(2+\epsilon)\sqrt{\nu(x)\log x} \right\},
\end{align}
which proves \eqref{roughapprox}.  We can thus assume without loss of generality in the following that a prime popular on $[2,x]$ must lie in the range where Hildebrand's approximation \eqref{hildapprox} holds, which we now use along with \eqref{alladiapprox} and \eqref{xiint} to prove \eqref{firstapprox}. 

Suppose $q$ is a prime lying in the interval \eqref{roughapprox}, also satisfying \begin{equation}
|\log P_0 - \log q|>2\nu(x). \label{roughlogapprox}
\end{equation}  We will show that for sufficiently large $x$, $\Psi\left(\frac{x}{P_0},P_0\right)>\Psi\left(\frac{x}{q},q\right)$ which means that some other prime occurs more frequently than $q$ as the largest prime divisor on the interval $[2,x]$, which will then imply \eqref{firstapprox} because $\nu(x) = O(\log \log x)$.  

Letting $u_q=\frac{\log x}{\log q}-1$ and, as before, $u_0=\frac{\log x}{\log P_0}-1 = \sqrt{\frac{\log x}{\nu(x)}}-1$, we have, using \eqref{alladiapprox} and \eqref{xiint} that
\begin{align}
\frac{\Psi\left(\frac{x}{P_0},P_0\right)}{\Psi\left(\frac{x}{q},q\right)} &= \frac{\frac{x}{P_0}\rho\left(\frac{\log x}{\log P_0}-1\right)}{\frac{x}{q}\rho\left(\frac{\log x}{\log q}-1\right)}\left(1+O\left(\frac{\log \log x}{\log P_0}\right)\right)\nonumber \\
&= \frac{q}{P_0}\sqrt{\frac{\xi'(u_0)}{\xi'(u_q)}}\exp\left\{\int_{u_0}^{u_q} \xi(t)dt\right\} \left(1+O\left(\frac{1}{u_0}\right) +O\left(\frac{\log \log x}{\log P_0}\right)\right). \label{approxratio}
\end{align}
First, if $q<P_0$ then $u_0<u_q$ and so, using \eqref{xiprime} we know that $\frac{\xi'(u_0)}{\xi'(u_q)} >1$, at least for sufficiently large $x$, and that $\int_{u_0}^{u_q}\xi(t)dt > (u_q-u_0)\xi(u_0)$.  Using these inequalities we see that the main term in \eqref{approxratio} is greater than
\begin{align}
\frac{q}{P_0}\exp\left\{(u_q-u_0)\xi(u_0)\right\}
&= \exp\left\{\log q - \log P_0 + \left(\frac{\log x}{\log q}-\frac{\log x}{\log P_0}\right)\xi(u_0)\right\}.  \label{psiratin}
\end{align}
Because $\xi(u_0) = \xi\left(\frac{\log x}{\log P_0} -1\right) = \nu(x) = \frac{\log^2 P_0}{\log x}$, we can rewrite the exponent above as \begin{equation}
\log q -\log P_0 + \frac{\log^2 P_0}{\log q} -\log P_0. \label{qexponent}
\end{equation}
Differentiating this with respect to $\log q$ gives $1-\frac{\log^2 P_0}{\log^2 q}$, which is negative for all $q<P_0$, and so, as a function of $q$, \eqref{qexponent} is strictly decreasing for all $q<P_0$.  Thus, in our situation, \eqref{psiratin} is minimized when $\log P_0 -\log q = 2\nu(x)$ in which case it equals
\begin{align*}
\exp\left\{-2\nu(x) +\log P_0\left(\frac{\log P_0}{\log P_0 -2\nu(x)} -1 \right)\right\}
&=\exp\left\{\frac{4\nu(x)^2}{\log P_0-2\nu(x)}\right\}.
\end{align*}
This is not only greater than 1, but also asymptotically greater than the error term of \eqref{approxratio}, and so we can conclude that the ratio there is strictly greater than 1, for sufficiently large $x$. Therefore some other prime occurs more frequently than $q$ as the largest prime divison on $[2,x]$.  

If instead, $q > P_0$, we have $u_q<u_0$ which means $\frac{\xi'(u_0)}{\xi'(u_q)} <1$, and so a little more care is required.  Let  $\delta = u_0-u_q$. Because $\log q - \log P_0> 2\nu(x)$, we will have that 
\begin{equation*}\delta =u_0-u_q =\frac{\log x}{\log P_0} - \frac{\log x}{\log q} >\frac{\log x}{\log P_0} - \frac{\log x}{\log P_0 + 2\nu(x)} > \frac{1.99\nu(x)\log x}{\log^2 P_0} = 1.99 \end{equation*}
for sufficiently large $x$.  Also, from \eqref{roughapprox}, we may assume that for any fixed $0<\epsilon<\frac{1}{2}$ and sufficiently large $x$, $ \log q <(2+\epsilon) \log P_0$ and so \begin{equation} \delta =\frac{\log x}{\log P_0} - \frac{\log x}{\log q}<u_0\left(1-\frac{1}{2+\epsilon}\right)<\frac{3u_0}{5}. \label{roughupper} \end{equation}  In this case we can use \eqref{xiprime} to approximate
\begin{align*}
\frac{\xi'(u_0)}{\xi'(u_q)}&=\left(\frac{\xi(u_0)}{\xi(u_0-\delta)}\right)\left(\frac{(u_0-\delta)\xi(u_0-\delta)-u_0+\delta+1}{u_0\xi(u_0)-u_0+1}\right)\nonumber \\
&=\left(\frac{\xi(u_0)}{\xi(u_0)+O\left(\delta\xi'(u_0)\right)}\right)\left(1-\frac{\delta\left(\xi(u_0)+O(1)\right)}{u_0\xi(u_0)-u_0+1}+O\left(\frac{1}{u_0\xi(u_0)}\right)\right)\nonumber \\
&=1-\frac{\delta}{u_0} + O\left(\frac{\delta}{u_0\xi(u_0)}\right).
\end{align*}
If we now use the somewhat more precise approximation \begin{equation*}\int_{u_0}^{u_q}\xi(t)dt > -(u_0-u_q)\xi(u_0)+\frac{1}{2}(u_0-u_q)^2\xi'(u_0),\end{equation*} we have that the main term of \eqref{approxratio} is greater than
\begin{align}
\frac{q}{P_0}\sqrt{\frac{\xi'(u_0)}{\xi'(u_q)}}\exp\left\{(u_q-u_0)\xi(u_0)+\frac{1}{2}(u_0-u_q)^2\xi'(u_0)\right\}. \label{qgreater}
\end{align}
We now consider the term 

\begin{align}
\sqrt{\frac{\xi'(u_0)}{\xi'(u_q)}}\exp&\left\{\frac{1}{2}(u_0-u_q)^2\xi'(u_0)\right\} \nonumber \\
 &=\sqrt{1-\frac{\delta}{u_0} + O\left(\frac{\delta}{u_0\xi(u_0)}\right)}\exp\left\{\frac{\delta^2}{2u_0}\left(1+O\left(\tfrac{1}{\xi(u_0)}\right)\right)\right\}, \label{mixedterm}
\end{align}
using \eqref{xiprime}. Using \eqref{roughupper} we see that when $\delta > \sqrt{u_0}$ this expression is greater than $\sqrt{\frac{2e}{5}}>1$ for sufficiently large $x$.  If $\delta \leq \sqrt{u_0}$, then we can rewrite \eqref{mixedterm} as

\begin{align}
\exp\left\{\frac{\delta^2}{2u_0} +\tfrac{1}{2}\log\left(1-\frac{\delta}{u_0}\right) + O\left(\frac{\delta^2}{u_0\xi(u_0)}\right)\right\}&= \exp\left\{\frac{\delta^2-\delta}{2u_0}+ O\left(\frac{\delta^2}{u_0\xi(u_0)}\right)\right\},\label{extraterm}
\end{align}
which is greater than 1 for sufficiently large $x$ since $\delta>2$.  Since this term \eqref{mixedterm} can now be ignored in inequality \eqref{qgreater} we are left with the same inequality \eqref{psiratin} as in the first case, and essentially the same analysis shows that the ratio is again greater than 1.  This proves equation \eqref{firstapprox}.

In order to prove the theorem, we will now need to use the more precise approximation \eqref{saiasappone}. In particular we have, using that $u = \frac{\log x}{\log p} - 1$ and that $p$ is in the interval described in \eqref{roughapprox},
\begin{align}
\Psi\left(\frac{x}{p},p\right) &= \frac{x}{p} \kappa\left(\frac{x}{p},p\right)\left(1+O\left(\left(\frac{\log \log x}{\log p}\right)^2\right)\right)\nonumber \\
&=\frac{x}{p}\left(\rho\left(u\right) +(\gamma-1)\frac{\rho'\left(u\right)}{\log p}\right)\left(1+O\left(\left(\frac{\log \log x}{\log p}\right)^2\right)\right)\nonumber \\
&=\frac{x}{p}\rho(u)\left(1+\frac{1-\gamma}{\log p}\xi(u)+
O\left(\frac{1}{u^2}\right)\right).\label{saiasone}
\end{align}

Let $c$ be a parameter, \begin{equation}
s = \exp\left\{\sqrt{\nu(x) \log x} + c\nu(x)\right\},\label{qdef}
\end{equation}
and $u_s = \frac{\log x}{\log s} -1$.

In the following we will optimize the value of $c$ as a function of $x$, however, from \eqref{roughlogapprox}, we can assume without loss of generality that $|c| \leq 2$.  In particular, we would like to choose $c$ so as to maximize the ratio \begin{align}
\frac{\Psi\left(\frac{x}{s},s\right)}{\Psi\left(\frac{x}{P_0},P_0\right)}&= \frac{\frac{x}{s}\rho(u_s)\left(1+\frac{1-\gamma}{\log s}\xi(u_s)\right)}{\frac{x}{P_0}\rho(u_0)\left(1+\frac{1-\gamma}{\log P_0}\xi(u_0)\right)}\left(1+O\left(\frac{1}{u^2_0}\right)\right)\nonumber \\
&=e^{-c\nu(x)}\frac{\rho(u_s)\left(1+\frac{1-\gamma}{ \sqrt{\nu(x)\log x}+c\nu(x)}\xi(u_s)\right)}{\rho(u_0)\left(1+\frac{1-\gamma}{\sqrt{\nu(x)\log x}}\xi(u_0)\right)}\left(1+O\left(\frac{1}{u^2_0}\right)\right). \label{psirat}
\end{align}

Now, since $\sqrt{\frac{\nu(x)}{\log x}} = \frac{1}{u_0+1}$, we have that

\begin{align}
u_s-u_0 = \frac{\log x}{\log s} - \frac{\log x}{\log P_0} &= \frac{\log x}{\sqrt{\nu(x)\log x}+c\nu(x)} - \frac{\log x}{\sqrt{\nu(x)\log x}}\nonumber \\
&=-c+c^2\sqrt{\frac{\nu(x)}{\log x}}+O\left(\frac{c^3}{u_0^2}\right)\nonumber \\
&= -c +\frac{c^2}{u_0+1}+O\left(\frac{c^3}{u_0^2}\right) \label{udiff}
\end{align}
and so, using Lemma \ref{shift}, and the fact that $\nu(x) = \xi(u_0)$,
 \begin{align}
\frac{\rho(u_s)}{\rho(u_0)} &= \exp\left\{\nu(x)\left(c-\frac{c^2}{u_0{+}1}{+}O\left(\frac{c^3}{u_0^2}\right)\right)\right\}\left(1+\frac{c}{2u_0}\left(1-\frac{c\nu(x)}{\nu(x){-}1}+\frac{1}{(\nu(x){-}1)^2}\right)+O\left(\frac{1}{u_0^2}\right)\right)\nonumber \\
&= \exp\left\{c\nu(x){-}c^2\left(\frac{\nu(x)}{u_0{+}1}{+}\frac{\nu(x)}{2u_0(\nu(x){-}1)}\right) {+} \frac{c}{2u_0}\left(1{+}\frac{1}{(\nu(x){-}1)^2}\right){+}O\left(\frac{c^3\nu(x){+}1}{u_0^2}\right)\right\}.\label{secondpsiratterm}
\end{align}
Also, we see that the final term of \eqref{psirat} can be ignored since 
\begin{align}
\frac{1+\frac{1-\gamma}{ \sqrt{\nu(x)\log x}+c\nu(x)}\xi(u_s)}{1+\frac{1-\gamma}{\sqrt{\nu(x)\log x}}\xi(u_0)} &= 1+ \frac{\frac{1-\gamma}{ \sqrt{\nu(x)\log x}+c\nu(x)}\xi(u_s)-\frac{1-\gamma}{\sqrt{\nu(x)\log x}}\xi(u_0)}{\left(1+\frac{1-\gamma}{\sqrt{\nu(x)\log x}}\xi(u_0)\right)}\nonumber \\
&=1+ \frac{\frac{1-\gamma}{\sqrt{\nu(x)\log x}}\left(\xi(u_s)-\xi(u_0)\right) + O\left(\frac{\xi(u_s)}{\log x}\right)}{\left(1+\frac{1-\gamma}{\sqrt{\nu(x)\log x}}\xi(u_0)\right)}\nonumber \\
&=1+ \frac{\frac{1-\gamma}{\sqrt{\nu(x)\log x}}\xi'(u_0)\left(u_s-u_0\right) + O\left(\frac{1}{u_0^2}\right)}{\left(1+\frac{1-\gamma}{\sqrt{\nu(x)\log x}}\xi(u_0)\right)}\nonumber \\
&=1+O\left(\frac{1}{u_0^2}\right). \label{thirdpsiratterm}
\end{align}
Using \eqref{secondpsiratterm} and \eqref{thirdpsiratterm} in the ratio \eqref{psirat} we have that
\begin{align}
&\frac{\Psi\left(\frac{x}{s},s\right)}{\Psi\left(\frac{x}{P_0},P_0\right)}\nonumber \\
&=\exp\left\{-c^2\nu(x)\left(\frac{1}{u_0{+}1}{+}\frac{1}{2u_0(\nu(x){-}1)}\right) {+} \frac{c}{2u_0}\left(1{+}\frac{1}{(\nu(x){-}1)^2}\right){+}O\left(\frac{c^3\nu(x){+}1}{u_0^2}\right)\right\}, \label{combinedrat}
\end{align}
and so maximizing this ratio is equivalent to maximizing the polynomial expression in $c$ appearing in the exponent.  After rescaling by dividing out a factor of $\nu(x)\left(\frac{1}{u_0+1}+\frac{1}{2u_0(\nu(x)-1)}\right)$ this expression is  \begin{align}
-c^2 + \frac{c}{2\nu(x)}&\left(\frac{1+\frac{1}{(\nu(x)-1)^2}}{\frac{u_0}{u_0+1}+\frac{1}{2(\nu(x)-1)}}\right)+O\left(\frac{c^3}{u_0}+\frac{1}{\nu(x)u_0}\right)
\end{align}
which is maximized by some $c$ satisfying 
\begin{align}
c &= \frac{1}{4\nu(x)}\left(\frac{1+\frac{1}{(\nu(x)-1)^2}}{\frac{u_0}{u_0+1}+\frac{1}{2(\nu(x)-1)}}\right)+O\left(\frac{c^{3/2}}{\sqrt{u_0}}+\frac{1}{\sqrt{\nu(x)u_0}}\right)\nonumber\\
&=\frac{1}{2\nu(x)}\left(\frac{\nu(x)^2-2\nu(x)+2}{2\nu(x)^2-3\nu(x)^1+1}\right)+O\left(\frac{1}{\sqrt{\nu(x)u_0}}\right)\nonumber \\
&= \frac{1}{4\nu(x)}\left(1-\frac{\nu(x)-3}{2\nu(x)^2 -3\nu(x)+1}\right)+O\left(\frac{1}{\sqrt{\nu(x)u_0}}\right).
\end{align}
Using this expression for $c$ in \eqref{qdef}, we see that the ratio \eqref{psirat} is maximized when $s$ satisfies the expression given in \eqref{massympt}.
\end{proof}

We can use this result to give an asymptotic for the number of times that a prime which is popular on $[2,x]$ appears as the largest prime divisor of an integer on that interval, which we denote by $C(x)$, thus giving the height of the peak of the distribution of $P(n)$ on the interval $[2,x]$.  (Note that if multiple primes are popular on $[2,x]$, they occur the same number of times on that interval, so the function $C(x)$ is well defined for all $x$.)  This theorem is Theorem \ref{intro:modeheight} in the introduction.
 
\begin{theorem} \label{cor:peakheight} If $p$ is popular on the interval $[2,x]$, then $C(x)$, the count of integers $n \in [2,x]$ for which $P(n)=p$, is given asymptotically by 
\begin{equation}
C(x) = \frac{x}{\sqrt{2\pi\log x}}\exp\left\{{-}2\sqrt{\nu(x)\log x}{+} \int\displaylimits_{0}^{\nu(x)}\frac{e^s {-}1}{s}ds{+}\frac{3\nu(x)}{2}{+}\gamma{+}O\left(\frac{1}{\nu(x)}\right)\right\}. \label{peakheightvalue}
\end{equation}
\end{theorem}
\begin{proof}
We know from the above theorem that if $p$ is popular on $[2,x]$ then 
\begin{equation}
p = \exp\left\{\sqrt{\nu(x) \log x }+\frac{1}{4}+O\left(\frac{1}{\nu(x)}\right)\right\} \label{papprox}.
\end{equation}
Using \eqref{hildapprox}, 
\begin{align}
\Psi\left(\frac{x}{p},p\right) 
&=\frac{x}{p}\rho(u)\left(1+
O\left(\frac{\log(1+ u)}{\log p}\right)\right), \label{peakcountapprox}
\end{align}
where  
\begin{align}u &= \frac{\log x}{\log p} -1 \ =\ \frac{\log x}{\sqrt{\nu(x) \log x }+\frac{1}{4}+O\left(\frac{1}{\nu(x)}\right)}-1 \nonumber \\
&= \sqrt{\frac{\log x}{\nu(x)}} -1-\frac{1}{4\nu(x)}+O\left(\frac{1}{\nu(x)^2}\right). \label{peaku}
\end{align}
Now, 
\begin{align}
\xi(u) &= \xi\left(\sqrt{\frac{\log x}{\nu(x)}} -1+O\left(\frac{1}{\nu(x)}\right)\right)\nonumber \\
&= \xi\left(\sqrt{\frac{\log x}{\nu(x)}} -1\right) + O\left(\frac{\xi'\left(\sqrt{\frac{\log x}{\nu(x)}} -1\right)}{\nu(x)}\right)\nonumber \\
&=\nu(x)+ O\left(\frac{1}{\sqrt{\nu(x)\log x}}\right),
\end{align} 
so, using \eqref{alladiapprox}, along with \eqref{xiprime} we see that
\begin{align}
\rho(u) &= \left(1+O\left(\frac{1}{u}\right)\right)\sqrt{\frac{\xi'(u)}{2\pi}} \exp\left\{\gamma -u\xi(u) + \int_{0}^{\xi(u)}\frac{e^s -1}{s}ds\right\}\nonumber \\
&= \left(1+O\left(\frac{1}{\nu(x)}\right)\right)\sqrt{\frac{1}{2\pi u}} \exp\left\{\gamma -u\nu(x) +\int_{0}^{\nu(x)}\frac{e^s -1}{s}ds + O\left(\frac{e^{\nu(x)} -1}{\nu(x)\sqrt{\nu(x)\log x}}\right)\right\}\nonumber \\
&=\frac{1}{\sqrt{2\pi}}\left(\frac{\nu(x)}{ \log x}\right)^{1/4} \exp\left\{\gamma -\sqrt{\nu(x)\log x}+\nu(x) +\frac{1}{4}+\int_{0}^{\nu(x)}\frac{e^s -1}{s}ds+O\left(\frac{1}{\nu(x)}\right)\right\}.  \label{eq:maxrhou}
\end{align}
Combining this with \eqref{papprox} and \eqref{peakcountapprox} we have that 
\begin{align}
\Psi\left(\tfrac{x}{p},p\right) = \frac{x}{\sqrt{2\pi\log x}}\exp\left\{-2\sqrt{\nu(x)\log x}{+} \int\displaylimits_{0}^{\nu(x)}\frac{e^s {-}1}{s}ds{+}\frac{3\nu(x)}{2}{+}\gamma {+}O\left(\frac{1}{\nu(x)}\right)\right\}, \nonumber
\end{align}
where we have also used \eqref{nufunceq} to see that $e^{\nu(x)/2}=(\nu(x)\log x)^{1/4} + O\left(\frac{\nu(x)^{3/4
}}{(\log x)^{1/4}}\right)$.
\end{proof}

Note that, asymptotically, $\int_{0}^{\nu(x)} \frac{e^s -1}{s}ds = \frac{e^{\nu(x)}}{\nu(x)}+O\left(\frac{e^{\nu(x)}}{\nu(x)^2}\right)=\sqrt{\frac{\log x}{\nu(x)}}+O\left(\frac{\sqrt{\log x}}{\nu(x)}\right)$, and so the expression in \eqref{peakheightvalue} is given approximately by \begin{equation}
x\exp\left\{-\sqrt{2\log x \left(\log \log x + \log \log \log x - (2+\log 2) + o(1)\right)}\right\}
\end{equation} 
which is the estimate given in \cite[Theorem 1]{koninckmode}.

\section{Popular primes} \label{popprimes}
Having seen that the value of any prime which is popular on the interval $[2,x]$ tends, slowly, to infinity and takes on prime values, one might expect that every prime number is popular on some such interval.  This turns out not to be the case.  We define a \textbf{popular prime} to be a prime number which is popular on some such interval $[2,x]$. 

In what follows we will see that not only are there prime numbers which are not popular, but in fact there is a positive proportion of primes which are not popular.  First however, we use Theorem \ref{llqest} to give a lower bound for their count.

\begin{corollary} \label{lbound}
There exists an absolute positive constant $C$ such that the count of of the popular primes up to $x$, for $x>10$, is at least $C\frac{\log^{3/2} x}{\sqrt{\log \log x}}$.
\begin{proof}
Theorem \ref{llqest} implies that there exists an absolute constant $C'$ such that for any popular prime, $p$, popular on the interval $[2,x']$, there exists another popular prime in the interval 
\begin{equation}\left(p,p\left(1+C'\left(\left(\frac{\log \log x'}{\log x'}\right)^{1/4}\right)\right)\right].\end{equation}
Setting $y=p$ we have $\log y = \sqrt{\log x'\log \log x'}+O\left(1\right)$, and so we see that for a suitably large choice of  $C''$ and any $y$ there is a popular prime in the interval 
\begin{equation} \left(y,y\left(1+C''\sqrt{\frac{\log \log y}{\log y}}\right)\right]. \label{popint} \end{equation}
If we restrict to counting popular primes appearing in $[x^{1/2},x]$ in intervals of the form \eqref{popint} where $y$ is greater than $x^{1/2}$, then we may assume that $1+C''\sqrt{\frac{\log \log y}{\log y}}\geq  1+C'''\sqrt{\frac{\log \log x}{\log x}}$ for yet another constant $C'''$. The number of non-overlapping intervals of the form $\left(y,y\left(1+C'''\sqrt{\frac{\log \log x}{\log x}}\right)\right]$ between $x^{1/2}$ and $x$ is \begin{equation}\frac{\tfrac{1}{2}\log x}{\log\left(1+C'''\sqrt{\frac{\log \log x}{\log x}}\right)} \gg \frac{\log^{3/2} x}{\sqrt{\log \log x}}\end{equation}  and the result follows.
\end{proof}
\end{corollary}

Before we can prove an upper bound for the distribution of the popular primes, we need first a version of the Buchstab identity for the function $\Psi(x,y)$ defined earlier.

\begin{lemma} \label{genbuch} Let $p_n$ denote the $n$th prime number.  For any $k \geq 1$, \begin{equation}\Psi\left(\frac{x}{p_{n+k}},p_{n+k}\right) = \Psi\left(\frac{x}{p_{n+k}},p_n\right)+\sum_{i=1}^k \Psi\left(\frac{x}{p_{n+k}p_{n+i}},p_{n+i}\right).\end{equation}
\end{lemma}

\begin{proof}  The left hand side, $\Psi\left(\tfrac{x}{p_{n+k}},p_{n+k}\right)$ counts those integers at most $x$ whose largest prime factor is $p_{n+k}$.  Taking such an integer $m$, and dividing out a factor of $p_{n+k}$ we obtain an integer, $\tfrac{m}{p_{n+k}}$, at most $\tfrac{x}{p_{n+k}}$ whose largest prime factor is either less than or equal to $p_n$, in which case $m$ is counted by $\Psi\left(\tfrac{x}{p_{n+k}},p_n\right)$, or its largest prime factor is $p_{n+i}$ for some $1 \leq i \leq k$, in which case $m$ is counted by $\Psi\left(\tfrac{x}{p_{n+k}p_{n+i}},p_{n+i}\right)$.
\end{proof}

We can use this lemma to show that the average prime spacing between popular primes cannot be too small.


\begin{theorem} \label{spacingthm}If the primes $p_{n}$ and $p_{n+k}$ are any two popular primes satisfying \begin{equation}
p_{n+k}-p_n=O\left(\frac{p_n}{\log p_n}\right), \label{maxspace}
\end{equation} then the average prime gap between these primes must satisfy \begin{equation}\frac{p_{n+k}-p_{n}}{k} \geq \left(1+O\left(\frac{\log \log p_n}{\log p_n}\right)\right) \frac{\rho(2-\alpha)\log p_{n}}{2-\alpha}, \label{spacing} \end{equation} 
where $\alpha = \frac{\log(p_{n+k}-p_n)}{\log p_n}$. 
\end{theorem}
\begin{proof}

Suppose that $p_n$ and $p_{n+k}$, $k >0$ are any two popular primes satisfying \eqref{maxspace} and let $\alpha = \frac{\log(p_{n+k}-p_n)}{\log p_n}$.  Because both $p_n$ and $p_{n+k}$ are popular, there exist integers $x_n$ and $x_{n+k}$ such that $p_n$ is popular on the interval $[2,x_n]$, and likewise $p_{n+k}$ is popular on $[2,x_{n+k}]$.

Now, as $x$ increases, the function $\Psi\left(\frac{x}{p},p\right)$ is nondecreasing, in fact, as $x$ increases through the integers, the difference $\Psi\left(\frac{x+1}{p},p\right)-
\Psi\left(\frac{x}{p},p\right)$ is either 0 or 1. So, in the case that $x_{n+k}\geq x_n$, we have that \begin{equation}
\Psi\left(\frac{x_n}{p_{n+k}},p_{n+k}\right)\leq\Psi\left(\frac{x_n}{p_n},p_n\right)\leq\Psi\left(\frac{x_{n+k}}{p_n},p_n\right)\leq \Psi\left(\frac{x_{n+k}}{p_{n+k}},p_{n+k}\right).
\end{equation}
Thus, we see that as $x$ increases from $x_n$ to $x_{n+k}$, there must be an intermediate integer $x'$ between $x_{n+k}$ and $x_{n}$ for which 
\begin{equation} \label{xprime}
\Psi\left(\frac{x'}{p_{n+k}},p_{n+k}\right) = \Psi\left(\frac{x'}{p_n},p_n\right). 
\end{equation}  
Note that it need not necessarily be the case that $x_{n+k}\geq x_n$, however the case that $x_{n+k}< x_n$ is essentially identical and we again find an integer $x'$ between these values satisfying \eqref{xprime}.

By Theorem \ref{llqest} we know that both 
\begin{equation*}
\log p_n  = \sqrt{\nu(x_n)\log x_n}+\frac{1}{4}+O\left(\frac{1}{\nu(x_n)}\right) \label{pnapprox}
\end{equation*}
and 
\begin{equation*}
\log p_{n+k} = \sqrt{\nu(x_{n+k})\log x_{n+k}}+\frac{1}{4}+O\left(\frac{1}{\nu(x_{n+k})}\right). \label{pnkapprox}
\end{equation*}
Since $\log p_{n+k}-\log p_n  = O\left(\frac{1}{ \log p_n }\right)$ and $x'$ lies between $x_n$ and $x_{n+k}$ we must have that 
\begin{align}
\log p_{n} &= \sqrt{\nu(x')\log x'}+\frac{1}{4}+O\left(\frac{1}{\nu(x')}\right). \label{pnxprime}
\end{align} Set $u_0 = \frac{\log x'}{\log p_{n}}-1$. Using Equation \ref{xprime}, Lemma \ref{genbuch} and the approximation $\Psi(x,y) = \left(1{+}O\left(\frac{1}{u}\right)\right)x\rho(u)$ we can write 
\begin{align}\Psi\left(\frac{x'}{p_n},p_n\right) - \Psi\left(\frac{x'}{p_{n+k}},p_{n}\right)&=\sum_{i=1}^k \Psi\left(\frac{x'}{p_{n+k}p_{n+i}},p_{n+i}\right) \nonumber \\
&= \left(1+O\left(\frac{1}{u_0}\right)\right)\sum_{i=1}^k \frac{x'}{p_{n+i}p_{n+k}}\rho\left(\frac{\log x - \log p_{n+k}-p_{n+i}}{\log p_{n+i}}\right).
\end{align}

Using Lemma \ref{shift}, 

\begin{align}
\rho\left(\frac{\log x' -\log p_{n+k}-\log p_{n+i}}{\log p_{n+i}}\right) &= \rho\left(\frac{\log x' -2\log p_{n} +O\left(\frac{1}{\log p_n}\right)}{\log p_{n}\left(1+O\left(\frac{1}{\log^2 p_n }\right)\right)}\right) \nonumber \\
&= \rho\left((u_0-1)\left(1+O\left(\frac{1}{\log^2 p_n }\right)\right)\right) \nonumber \\
&=\left(1+O\left(\frac{u_0 \log u_0 }{\log^2 p_n}\right)\right)\rho(u_0-1)\nonumber \\
&=\left(1+O\left(\frac{1}{\log p_n}\right)\right)\rho(u_0-1).
\end{align}

Since
\begin{equation}\sum_{i=1}^k \frac{1}{p_{n+i}} = \sum_{i=1}^k \frac{1}{p_{n}+O\left(\frac{p_n }{\log p_n}\right)} = \frac{k}{p_{n}}\left(1+ O\left(\frac{1}{\log p_n}\right)\right),\end{equation}
we have that 
\begin{align}\Psi\left(\frac{x'}{p_n},p_n\right) - \Psi\left(\frac{x'}{p_{n+k}},p_{n}\right)&= \left(1+O\left(\frac{1}{u_0}\right)\right)\frac{x'k}{p_np_{n+k}}\rho(u_0-1). \label{buchdiff}
\end{align}

On the other hand, using Hildebrand's upper bound for the count of smooth numbers in short intervals with $z = \frac{x'(p_{n+k}-p_n)}{p_np_{n+k}} =\frac{x'p_n^{\alpha -1}}{p_{n+k}}$ we have that 
\begin{align}
\Psi\left(\frac{x'}{p_n},p_n\right) &- \Psi\left(\frac{x'}{p_{n+k}},p_{n}\right) = \Psi\left(\frac{x'}{p_{n+k}} +z,p_n\right) - \Psi\left(\frac{x'}{p_{n+k}},p_{n}\right)\nonumber \\
&\leq\left(1+O\left(\frac{1}{\log p_{n+k}}\right)\right) \frac{\Psi\left(\frac{x'}{p_{n+k}},p_n\right)p_n\log\left(\frac{x'p_n}{zp_{n+k}}\right)}{\Psi\left(\frac{x'p_n}{zp_{n+k}},p_n\right)\log p_n} \nonumber \\
&=\left(1+O\left(\frac{1}{u_0}\right)\right)\frac{\frac{x'p_n}{p_{n+k}}\rho\left(\frac{\log x'-\log p_{n+k}}{\log p_n}\right)\left(\log x' - \log z +O\left(\frac{1}{\log p_n}\right)\right)}{\frac{x'p_n}{zp_{n+k}}\rho\left(\frac{\log x' + \log p_n- \log p_{n+k} -\log z}{\log p_n}\right)\log p_n} \nonumber \\
&=\left(1+O\left(\frac{1}{u_0}\right)\right)\frac{z\rho\left(u_0\right)\left((2-\alpha)\log p_n +O\left(\frac{1}{\log p_n}\right)\right)}{\rho\left(\frac{(2-\alpha)\log p_n +O\left(\frac{1}{\log p_n}\right)}{\log p_n}\right)\log p_n}\nonumber \\
&=\left(1+O\left(\frac{1}{u_0}\right)\right)\frac{(2-\alpha)\rho\left(u_0\right)x'(p_{n+k}-p_n)}{\rho\left(2-\alpha\right)p_np_{n+k}}.
\label{localdiff}
\end{align}

Using \eqref{xiprime} to see that 
\begin{align}
\xi(u_0) = \xi\left(\frac{\log x'}{\log p_n} -1\right) &=  \xi\left(\frac{\log x'}{\sqrt{\nu(x')\log x'}+O(1)}-1\right) \nonumber \\
&= \xi\left(\sqrt{\frac{\log x'}{\nu(x')}} -1+ O\left(\tfrac{1}{\nu(x')}\right)\right) \nonumber \\
&= \nu(x') + O\left(\frac{\xi'(u_0)}{\nu(x')}\right) = \nu(x') +O\left(\tfrac{1}{\log p_n}\right), \label{xinuapprox}
\end{align}
and, from the functional equation \eqref{nufunceq} for $\nu(x)$, that 
\begin{align}
\nu(x') = \log\left(1+\sqrt{\nu(x')\log x'} -\nu(x')\right) &= \log\left(\log p_n -\nu(x') + O(1)\right))\nonumber \\
&=\log \log p_n +o(1) \label{nuapprox}
\end{align}
we can conclude, by combining \eqref{buchdiff} and \eqref{localdiff}, and using \eqref{simpshift} that
\begin{align}
\frac{p_{n+k}-p_n}{k} &\geq \left(\frac{\rho(2-\alpha)}{2-\alpha}+O\left(\frac{1}{u_0}\right)\right)\frac{\rho(u-1)}{\rho(u_0)}\nonumber \\
&= \left(\frac{\rho(2-\alpha)}{2-\alpha}+O\left(\frac{1}{u_0}\right)\right)u_0\xi(u_0) \nonumber \\
&=\left(\frac{\rho(2-\alpha)}{2-\alpha}\right)\frac{\xi(u_0) \log x'}{\sqrt{\nu(x')\log x'}} +O(\xi(u_0)) \nonumber \\
&= \left(1+O\left(\frac{\log \log p_n}{\log p_n}\right)\right)\frac{\rho(2-\alpha)\log p_n}{2-\alpha}. \label{ppspacing}
\end{align} 
\end{proof}

As a corollary, we see that for any sufficiently large pair of twin primes, or consecutive primes with any fixed gap, the smaller of the pair will never be a popular prime.  In fact, approximating $\rho(2)/2 = 0.153\ldots$ we have the following stronger result, which is Theorem \ref{intro:gpy} in the introduction.

\begin{corollary}
Given any two sufficiently large consecutive primes, $p<q$, if the gap between them, $q-p$, is less than $0.153\log p$, then $p$ is not a popular prime.
\end{corollary}

Goldston, Pintz and Y{\i}ld{\i}r{\i}m \cite{GPY} have shown that for any fixed $\eta$, there is a positive proportion of prime numbers, $p$, which are followed by a gap less than $\eta \log p$, which means we can conclude the following, Corollary \ref{intro:pospro} from the introduction, as well. 

\begin{corollary}
A positive proportion of the prime numbers are not popular.
\end{corollary}

Note that if we assume that the smooth numbers are regularly distributed in all of the short intervals that we are concerned with in the proof of Theorem \ref{spacingthm} we can do much better.  Assuming, as is widely conjectured, that \begin{equation}
\Psi(x+z,y)-\Psi(x,y) \sim \frac{z}{x}\Psi(x,y)
\end{equation}
for $y\sim\exp\left(\sqrt{\nu(x)\log x}+\frac{1}{4}\right)$ and $z >x/y^2$, we could show, by the method of Theorem \ref{spacingthm}, that the average gap between any two popular primes $p$ and $q$, $p<q$, must be asymptotically equal to $\log q$, and thus that the popular primes have relative density 0 among the primes.

\section{Computations and the Convex Primes} \label{data}
Compiling a list of the popular primes is computationally difficult, as it requires counting all of the largest prime divisors of integers up to relatively large values of $x$ compared to the popular primes themselves. The first few popular primes (popular on some interval $[2,x]$ for some $x \leq $70,000,000,000,000) and the integer $x$ for which they
were first popular on the interval $[2,x]$ are given in the table below.
Note that thus far no prime has been a popular
prime without being the uniquely popular prime on
some such interval.  Further, the table gives the
count of the number of times the prime occurs as the largest prime divisor of an integer in the interval $[2,x]$.

\medskip
\begin{center}
\textbf{Primes popular on some interval $[2,x]$ for $x \leq  10^{14}$ }\nopagebreak \\
\smallskip
\begin{tabular}{|c|c|c|c|c|c|}
\hline
\shortstack{Popular\\Prime} & \shortstack{First popular\\ on $[2,x]$} & \shortstack{First uniquely\\popular} & \shortstack{Last popular\\on $[2,z]$ } &\begin{minipage}{5mm}$C(x)$\\\vspace{-1mm}\end{minipage}   &  \begin{minipage}{5mm}$C(z)$\\\vspace{-1mm}\end{minipage}  \\
\hline 
2	&	2	&	2	&   17      &	1   &   4\\
3	&	3	&	12	&   119     &	1   &   14\\
5	&	45	&	80	&   279 	&   8   &   25\\
7	&	70	&	196	&   1858    &	10  &   77\\
13	&	1456	&	1638	&   5471    &	67  &   151\\
19	&	4845	&	4864	&	29301   &   140 &   428\\
23	&	20332	&	22425	&	53474   &   344 &   616\\
31	&	46345	&	46500	&	117303  &   563 &   1005\\
43	&	106812	&	109779	&	220523  &   947 &   1517\\
47	&	153032	&	158625	&	611374  &   1197    &   2902\\
73	&	592760	&	603564	&	2642391 &   2846    &   7664\\
83	&	2484190	&	2552416	&	2672025 &   7357    &   7722\\
109	&	2620033	&	2620142	&	2952463 &   7621    &   8284\\
113	&	2623860	&	2627250	&	41192601    &   7629    &   48380\\
199	&	41163150	&	41163747	&   237611044   &	48357   &   161644\\
283	&	237321819	&	237398795	&	1967277194  &   161507  &   698074\\
467	&	1966462280	&	1966466950	&	13692930957 &   697875  &   2761234\\
661	&	13690728506 &   13690729828	&	64358549949 &   2760913 &   8357693\\
773	&	64322151699 &   64322158656	&	79880100420 &   8354317 &   9758410\\
887	&	79838726306 &   79838739611	&	220369251374    &   9754751 &   20285553\\
1109	&	220355977754    &   220355987735	&   232880841877    &	20284680    &   21123128\\
1129	&	232268764689    &   232268774850	&   618765808209    &	21082412    &   43031555\\
1327	&	618745965579    &   618745972214	&	1882062587041   &   43030537    &   96835113\\
1627	&	1882062393429   &   1882062476406	&	9607847299025   &   96835105    &   318539488\\
2143	&	9607711921430   &   9607713772982	&	19364476224949  &   318536223   &   534261087\\
2399	&	19364051434020  &   19364051829855	&	26396066576762  &   534252383   &   672081919\\
2477	&	26393150922356  &   26393150937218	&	37636861534247  &   672026918   &   873949289\\
2803	&	37636607775855  &   37636607806688	&	84128837898779  &   873944930   &   1588958920\\
2861    &   84128837864448  &   84128837898780  &   85992223800357  &   1588958920  &   1612740571\\
2971    &   85992223734996  &   85992223800358  &   89487767416445     &   1612740571  & 1656313907  \\
3023    &   89487767413423  &   89487767416446  &   90749798232275  &   1656313907 &    1672851087\\
3041    &   90749798153210  &   90749798232276  &   91157523869191  &   1672851087 &    1678444884\\
3049    &   91015395545226  &   91015395548275  &   91473520711546  &   1676495503 &    1682728352\\
3089    &   91473520705369  &   91473520711547  &   92913565436551  &   1682728352 &    1699108828\\
3137    &   93871134565472  &   93871134606253  &   94131107722837  &   1708870682  &   1712113344\\
3373    &   94131107675616  &   94131107722838  &   $>10^{14}$        &   1712113344  &   $>1791544685$ \\
\hline
\end{tabular}
\smallskip
\end{center}

Note that the ranges of popularity for 73, 83, 109 and 113 all overlap, and in fact all four are popular on the interval $[2,2626355]$, each occurring 7634 times.

Thus far, the data for the popular primes appear to be related to a subset of the prime numbers studied by Pomerance \cite{PNG} and Tutaj \cite{Tu} and also discussed in Guy's book of unsolved problems in number theory \cite[Problem A14]{guy}.   This set, the \lq\lq convex primes,'' is the set of those prime numbers numbers, $p_n$, which form the vertices of the boundary of the convex hull of the points $(n,p_n)$ in the plane.   Pomerance uses this set of primes to show that there are infinitely many primes $p_n$ which satisfy the inequality \begin{equation*}2p_n < p_{n-i} + p_{n+i} \hspace{5mm} \text{ for all positive $i < n$.}\end{equation*}

Using the best known error term for the prime number theorem, Pomerance claims that there are at least $\exp(c(\log x)^{3/5-\epsilon})$ convex primes up to $x$ for any $\epsilon>0$ and some constant $c>0$.  Assuming the Riemann hypothesis gives at least $c'x^{1/4}/\log^{3/2} x$ convex primes. 

The values of the popular primes computed above are a superset of the convex primes: all of the convex primes less than 3000 are also popular.  Furthermore, all of those primes, $p_n$, where the point $(n,p_n)$ lies on the boundary of the convex hull but is not a vertex point of it (namely  5, 13, 23, 31 and 43) are popular as well.  The popular primes $83,$ $109,773,1109,2143,2399,2477,2861,2971,3023,$ $3041,3049,3089,3137$ and 3373 correspond to points on the interior of the convex hull, however.  

Both convex primes and popular primes are more likely to be found after a run of densely packed primes, and prior to a larger than average gap betwen primes, which partially justifies the connection.  If one assumes that the convex primes continue to be a subset of the popular primes, then we would expect the count of the popular primes up to $x$ to be at least $x^{1/4}/\log^{3/2} x$, substantially better than what we are able to prove in Corollary \ref{lbound}.   In a forthcoming paper we will further discuss the convex primes, including a significantly improved upper bound for their count.

\section{Optimization of factoring algorithms: making squares} \label{sec:fast}
As mentioned in the introduction, the analysis done here is closely related to a key step in the analysis of the running time of a variety of factoring algorithms.  In particular,  one wishes to choose an optimal smoothness bound $y$ so as to minimize the number of random integers that must be chosen from the interval $[1,x]$ before the product of some subset of the integers chosen at random is a square.  When some subset of the integers has this property we say that the set has a \textit{square dependence.}  Since the probability an integer chosen at random from the interval $[1,x]$ is $y$-smooth is $\frac{x}{\Psi(x,y)}$, and any set of $\pi(y)+1$ $y$-smooth integers contains a square dependence, it is advantageous to pick a value of $y$ which minimizes the expression $\frac{x\pi(y)}{\Psi(x,y)}$, or equivalently maximizes \begin{equation}
\frac{\Psi(x,y)}{\pi(y)} = \left(1+O\left(\frac{1}{\log y}\right)\right)\frac{\Psi(x,y)\log y}{y} \approx \frac{\Psi(x,y)}{y}.
\end{equation}

The analysis of the maximum value of $\frac{\Psi(x,y)}{y}$ is highly similar to the analysis of the peak value of $\Psi\left(\frac{x}{p},p\right)$ performed in Section \ref{results}.  In fact, maximizing $\frac{\Psi(x,y)}{y}$ requires maximizing the same expression \eqref{saiasone} as in the proof of Theorem \ref{llqest}, with the modification that now $u = \frac{\log x}{\log p}$, rather than that value shifted by one.  One thus finds that after suitably modifying the implicitly defined function $\nu(x)$ used in the proof, replacing it instead with the function $\omega(x) = \xi\left(\sqrt{\frac{\omega(x)}{\log x}}\right)$, which satisfies the functional equation \begin{equation}
e^{\omega(x)} = 1+\sqrt{\omega(x)\log x}, 
\end{equation}
and, like $\nu(x)$ is given approximately by
\begin{equation}
\omega(x) = \tfrac{1}{2}\log \log x + \tfrac{1}{2}\log \log \log x -\tfrac{1}{2}\log 2 + o(1) 
\end{equation} as $x \to \infty$, the exact same analysis goes through and one obtains the following.
\begin{theorem} \label{thm:roughfast}
If, for a given value of $x$, the prime $p$ maximizes the expression $\frac{\Psi(x,p)}{p}$, then 
\begin{equation}
p = \exp\left\{\sqrt{\omega(x)\log x}+\frac{1}{4}\left(1-\frac{\omega(x)-3}{2\omega(x)^2 -3\omega(x)+1}\right)\right\}\left(1+ O\left(\left(\frac{\log \log x}{\log x}\right)^{1/4}\right)\right)\label{roughfastasym}. \end{equation}
\end{theorem}

Comparing the functions $\nu(x)$ and $\omega(x)$, we find that \begin{align}
\nu(x)-\omega(x) &= \log(\sqrt{\nu(x)\log x}-\nu(x)+1) - \log(\sqrt{\omega(x)\log x}+1)\nonumber \\
&=\frac{1}{2}\log \nu(x) - \frac{1}{2} \log \omega(x) + \log\left(1 - \sqrt{\frac{\nu(x)}{\log x}}\right)  + O\left(\frac{1}{\sqrt{\log x \log \log x}}\right)\nonumber \\
&= -\sqrt{\frac{\nu(x)}{\log x}} + \frac{1}{2}\log\left(1 + \frac{\nu(x) - \omega(x)}{\omega(x)}\right) + O\left(\frac{1}{\sqrt{\log x \log \log x}}\right)\nonumber \\
&= -\sqrt{\frac{\nu(x)}{\log x}}  + O\left(\frac{1}{\sqrt{\log x \log \log x}}\right). \label{vw}
\end{align}
We can use this to restate Theorem \ref{thm:roughfast} in terms of the function $\nu(x)$ for comparison to Theorem \ref{llqest}.
\begin{corollary}
If, for a given value of $x$, the prime $p$ maximizes the expression $\frac{\Psi(x,p)}{p}$, then 
\begin{equation}
p = \exp\left\{\sqrt{\nu(x)\log x}+\frac{3}{4} + O\left(\frac{1}{\log \log x}\right)\right\}. \end{equation}
\end{corollary}
\begin{proof}
Using \eqref{vw}, we see that 
\begin{align}
\sqrt{\omega(x)\log x} &= \sqrt{\left(\nu(x)+\sqrt{\frac{\nu(x)}{\log x}} + O\left(\frac{1}{\sqrt{\log x \log \log x}}\right)\right) \log x}\nonumber \\
&= \sqrt{\nu(x)\log x + \sqrt{\nu(x)\log x} + O\left(\sqrt{\frac{\log x}{\log \log x}}\right)}\nonumber \\
&= \sqrt{\nu(x)\log x} +\frac{1}{2} + O\left(\frac{1}{\log \log x}\right). \label{wvsqrtlogterm}
\end{align}
Using this approximation in \eqref{roughfastasym} the result follows. \end{proof}

The method of proof can also be adapted to maximize the function $\frac{\Psi(x,y)}{\pi(y)}$, which is slightly more relevant to the optimization of these factoring algorithms. Using the approximation $\pi(y) = \frac{y}{\log y}\left(1+\frac{1}{\log y} + O\left(\frac{1}{\log^2 y}\right)\right)$ we find that again, the analysis is nearly identical to that of Theorem \ref{llqest} with the function $\omega(x)$ used in place of $\nu(x)$.  However, instead of equation \eqref{psirat}, we find that we are maximizing the ratio

\begin{align}
\frac{\Psi\left(x,s\right)\pi(P_0)}{\Psi\left(x,P_0\right)\pi(s)}&= \frac{\frac{x \log s}{s}\rho(u_s)\left(1+\frac{1-\gamma}{\log s}\xi(u_s)\right)\left(1-\frac{1}{\log s}\right)}{\frac{x\log P_0}{P_0}\rho(u_0)\left(1+\frac{1-\gamma}{\log P_0}\xi(u_0)\right)\left(1-\frac{1}{\log P_0}\right)}\left(1+O\left(\frac{1}{u^2_0}\right)\right), 
\end{align}
where $u_0$ and $u_s$ have been suitably modified.

As before, the term \[ \frac{\left(1+\frac{1-\gamma}{\log s}\xi(u_s)\right)\left(1-\frac{1}{\log s}\right)}{\left(1+\frac{1-\gamma}{\log P_0}\xi(u_0)\right)\left(1-\frac{1}{\log P_0}\right)}\] can be absorbed into the error term, however the additional ratio of $\frac{\log s}{\log P_0}= \left(1+\frac{c}{u_0}\right)$ introduces an additional $\frac{c}{u_0}$ in the exponent of \eqref{combinedrat}.  

As a result, when we maximize $c$, we find that it now occurs for some $c$ satisfying \begin{equation}
c=\frac{3}{4\omega(x)} \left(1 - \frac{3\omega(x) -5}{6\omega(x)^2-9\omega(x)+3}\right) +O\left(\frac{1}{\sqrt{\omega(x)u_0}}\right).
\end{equation}
Thus we can conclude the following asymptotic, usefull in determining the optimal smoothness bound for use in integer factorization.
\begin{theorem}\label{thm:precfast}
If, for a given value of $x$, the prime $p$ maximizes the expression $\frac{\Psi(x,p)}{\pi(p)}$, then 
\begin{align}
p &= \exp\left\{\sqrt{\omega(x)\log x}+\frac{3}{4}\left(1-\frac{3\omega(x)-5}{6\omega(x)^2 -9\omega(x)+3}\right)\right\}\left(1+ O\left(\left(\frac{\log \log x}{\log x}\right)^{1/4}\right)\right) \nonumber \\ 
&= \exp\left\{\sqrt{\nu(x)\log x}+\frac{5}{4}+O\left(\frac{1}{\log \log x}\right)\right\}. \label{fastasymnu}
 \end{align}
\end{theorem}

Note that \eqref{fastasymnu} implies that in the limit as $x \to \infty$, the ratio of the prime, $p$, which maximizes $\frac{\Psi(x,p)}{\pi(p)}$ to a prime popular on $[2,x]$ tends to $e$.

Having estimated the value of $y$ which maximizes $\frac{\Psi(x,y)}{\pi(y)}$ relatively precisely, we can likewise give an estimate for the maximum value of this function.  Note that the maximum value of this function is what plays a key role in the analysis of factoring algorithms.  Denote by $h(x)$ this maximum value of $\frac{\Psi(x,y)}{\pi(y)}$ taken over all $y<x$. Croot, Granville, Pemantle and Tetali showed \cite{CGPT} that if one  chooses integers at random between 1 and $x$ until the sequence contains a square dependence, then the expected stoping time lies in the interval $\left[(\frac{\pi e^{-\gamma}}{4} +o(1))\frac{x}{h(x)},(e^{-\gamma} +o(1))\frac{x}{h(x)}\right]$, and futhermore that as $x \to \infty$, the stopping time lies, almost surely in this interval. The only estimate that they give for $h(x)$, however, is that $h(x)=x\exp\left\{-\sqrt{(2+o(1))\log x \log \log x}\right\}$. (In their notation, $J_0(x)=\frac{x}{h(x)}$.) We give here an asymptotic expression for the value of this function, proving Theorem \ref{intro:fastpeak} in the introduction.

\begin{theorem}
For a given value of $x$, the value of $h(x)$, the maximum value of $\frac{\Psi(x,y)}{\pi(y)}$ for $y<x$ is given asymptotically by 
\begin{equation}
h(x) = \frac{x}{\sqrt{2\pi\log x}}\exp\left\{-2\sqrt{\omega(x)\log x}+ \int_{0}^{\omega(x)}\frac{e^s {-}1}{s}ds+\frac{3\omega(x)}{2}+\gamma+O\left(\frac{1}{\log \log x}\right)\right\} 
\end{equation}
or, equivalently, the same expression with $\nu(x)$ in place of $\omega(x)$,
\begin{align}
h(x) &= \frac{x}{\sqrt{2\pi\log x}}\exp\left\{-2\sqrt{\nu(x)\log x}+ \int_{0}^{\nu(x)}\frac{e^s {-}1}{s}ds+\frac{3\nu(x)}{2}+\gamma+O\left(\frac{1}{\log \log x}\right)\right\} \nonumber \\
&= C(x)\left(1+O\left(\frac{1}{\log \log x}\right)\right), \label{fastvsmode}
\end{align}
where $C(x)$, defined before Corollary \ref{cor:peakheight}, is the number of times a prime, popular on $[2,x]$, appears as the largest prime divisor of an integer on that interval.
\end{theorem}

\begin{proof}
Because $\pi(y) = \frac{y}{\log y}\left(1+O\left(\frac{1}{\log y}\right)\right)$, the proof is essentially identical to that of Corollary \ref{cor:peakheight}, (again using $\omega(x)$ in place of $\nu(x)$) with the exception that in \eqref{peaku} we now have $u = \frac{\log x}{\log p}$, which causes us to lose a factor of $\omega(x)$ in the exponent of the expression \eqref{eq:maxrhou}, and that the final expression is multiplied by a factor of \[\log y = \sqrt{\omega(x)\log x}\left(1+O\left(\frac{1}{\sqrt{\omega(x)\log x}}\right)\right) = e^{\omega(x)} \left(1+O\left(\frac{1}{\sqrt{\omega(x)\log x}}\right)\right),\] which then restores that factor of $\omega(x)$ to the exponent. 

Using this, we obtain \eqref{fastvsmode} by using \eqref{wvsqrtlogterm} (which decreases the exponent by 1 when using $\nu(x)$) along with the observation that \begin{align*}
\int_{\nu(x)}^{\omega(x)} &\frac{e^s -1}{s} ds  = (\omega(x){-}\nu(x))\frac{e^{\nu(x)}{-}1}{\nu(x)} + O\left((\omega(x){-}\nu(x))\left(\frac{e^{\omega(x)}{-}1}{\omega(x)}{-}\frac{e^{\nu(x)}{-}1}{\nu(x)}\right)\right)\\
&=\left(\sqrt{\frac{\nu(x)}{\log x}}+O\left(\frac{1}{\sqrt{\nu(x)\log x}}\right)\right)\left(\sqrt{\frac{\log x}{\nu(x)}} -1\right) + O\left(\sqrt{\frac{\nu(x)}{\log x}}\left(\sqrt{\frac{\log x}{\omega(x)}}-\sqrt{\frac{\log x}{\nu(x)}}\right)\right)\\
&=1+O\left(\frac{1}{\nu(x)}\right),
\end{align*}
which, in turn, increases the exponent by 1.
\end{proof}

\section*{Acknowledgements}
I would like to thank my advisor, Carl Pomerance, for his invaluable guidance and encouragement throughout the
development of this paper.  I would also like to thank Florian Luca for listening to early versions of the arguments presented here, Jean-Marie de Koninck for pointing me to his work related to my problem, Sary Drappeau for suggesting the use of Saias' work to improve Theorem \ref{llqest}, Jonathan Bober for discussions about computing the popular primes, and Robin Pemantle for his interest in this project.

\bibliographystyle{amsplain}
\bibliography{popprimes}														
\end{document}